\documentclass[11pt]{amsart}

\usepackage{amsmath}
\usepackage{amssymb}
\usepackage{graphicx}
\usepackage{geometry}
\usepackage{enumerate}  % 添加这个宏包
\newtheorem{theorem}{Theorem}[section]

\newtheorem{lemma}[theorem]{Lemma}
\newtheorem{proposition}[theorem]{Proposition}
\theoremstyle{definition}
\newtheorem{definition}[theorem]{Definition}

\newtheorem{remark}[theorem]{Remark}
\numberwithin{equation}{section}
\title[A random demiclosedness principle]{A random demiclosedness principle for random asymptotically nonexpansive mappings}
\author[Y. Y. Sun]{Yuanyuan Sun}

\address[Y. Y. Sun]{School of Mathematics and Statistics, Central South University,
	Changsha, China}
\email{{yuanyuansun1205@163.com}}

\author[T. X. Guo]{Tiexin Guo*}
\address[T. X. Guo]{School of Mathematics and Statistics, Central South University,
	Changsha {\rm 410083}, China}
\email{tiexinguo@csu.edu.cn}
\thanks{*Corresponding author}

\author[Q. Tu]{Qiang Tu}
\address[Q. Tu]{School of Mathematics and Statistics, Central South University,
	Changsha, China}
\email{qiangtu126@126.com}



\begin{document}
\maketitle

% 献词部分 - 更精美的格式

\begin{center}
	\normalsize{ \textit{Dedicated to Professor Hong-Kun Xu's 65th birthday.}}
\end{center}

\begin{abstract}
By making full use of the inherent connection  between 
the theory of random conjugate spaces and the theory of 
classical conjugate spaces, in
this paper we establish a random demiclosedness 
principle for a random asymptotically nonexpansive 
mapping, which generalizes Xu's classical  
demiclosedness principle from a uniformly convex Banach 
space to a complete random uniformly convex random 
normed module: let $(E,\|\cdot\|)$ be a complete random 
uniformly convex random normed module, $E^{*}$ the 
random conjugate space of $E$, $G$ an almost surely 
bounded closed $L^{0}$-convex subset of $E$ and $f: G \rightarrow G$  a random asymptotically nonexpansive 
mapping, then $(I-f)$ is random demiclosed at $\theta$, 
namely, for each sequence $\{x_{n}, n\in \mathbb{N}\}$ 
in $G$, if $\{x_{n}, n\in \mathbb{N}\}$  converges in 
$\sigma(E, E^{*})$ to $x$ and $\{(I-f)x_{n}, n\in \mathbb{N}\}$  converges to $\theta$, then $(I-f)x=\theta$, where $I$ denotes the identity 
operator on $E$ and $\sigma(E, E^{*})$ the random weak 
topology on $E$.
\end{abstract}

\maketitle
%==================脚注位置添加基金支持-开始======================
\makeatletter
\newcommand\blfootnote[1]{%
	\begingroup
	\renewcommand\thefootnote{}\footnote{#1}%
	\addtocounter{footnote}{-1}%
	\endgroup
}
\makeatother

\blfootnote{This work was supported by the National Natural Science Foundation of China (Grant Nos.12371141, 12426645, 12426654) and the Natural Science Foundation of Hunan Province of China (Grant No.2023JJ30642).}

%==================脚注位置添加基金支持-结束======================
%%%%%%%%%%%%%%%%%%%%%%%%%%%%%%%%%%%%%%%%%%%%%%%%%%%%%%%%%%%%%%%%%%%%%%%%%%%%%%%%

\section{Introduction}
\label{intro}
%% Labels are used to cross-reference an item using \ref command.
\par
The remarkable  Browder's demiclosedness principle \cite{B76} states that, for a nonempty bounded closed convex subset $G$ of a uniformly convex Banach space $(B, \|\cdot\|)$ and a nonexpansive mapping $f: G \rightarrow B$, $(I-f)$ is demiclosed at each $y\in B$, that is, for any sequence $\{x_{n}, n\in \mathbb{N}\}$ in $G$, if $x_{n}\rightarrow x$ weakly and $(I-f)(x_{n})\rightarrow y$ strongly, then $(I-f)(x)=y$. The principle is also valid in a Banach space satisfying Opial's condition. Since then, the work in \cite{B76} has attracted the attention of many scholars in the field of nonlinear analysis \cite{H80, H82}, for example, it was  generalized from uniformly convex Banach spaces to Banach spaces with the Browder-G\"{o}hde property by Lin \cite{L87}.

%-----------------------------------------------------------------------------------------------------------------------------------------

\par
A lot of scholars have considered the demiclosedness principle for asymptotically nonexpansive mappings. As is well known, one of the celebrated results in the theory of asymptotically nonexpansive mappings is Xu's demiclosedness principle \cite{X91}: let  $(B, \|\cdot\|)$ be a uniformly convex Banach space,  $C$  a nonempty bounded closed convex subset of $B$, $f: C \rightarrow C$  an asymptotically nonexpansive mapping, then $(I-f)$ is demiclosed at zero.  It is well known that the demiclosedness principle  plays a key role in studying the asymptotic behavior and ergodic behavior of asymptotically nonexpansive mappings \cite{KSV14,LTX95,TX92}. For example, based on the demiclosedness principle for asymptotically nonexpansive mappings, Kirk \cite{KS01} established the following result: Let $K$ be a bounded subset of a uniformly convex Banach space $(B, \|\cdot\|)$ and  $f:K \rightarrow K$  a nonexpansive mapping. If $(I-f)$ is demiclosed, then for any $x\in K$, any weak subsequential limit point of $\{g^{n}(x), n\in \mathbb{N}\}$ is a fixed point of $f$, where $g=\frac{I+f}{2}$.  Reich \cite{R80} showed  that: Let  $(B, \|\cdot\|)$ be a uniformly convex Banach space with a Fr\'{e}chet differentiable norm, $K$ a closed convex subset of $B$, $f: K\rightarrow K$ a nonexpansive mapping with a fixed point, and $\{c_{n}, n\in \mathbb{N}\}$ a sequence in the real number interval $[0,1]$ such that $\sum^{\infty}_{n=1}c_{n}(1-c_{n})=+\infty$. If $x_{1}\in K$ and $x_{n+1}=c_{n}f(x_{n})+(1-c_{n})x_{n}$ for any $n\in \mathbb{N}$, then $\{x_{n}, n\in\mathbb{N}\}$ converges weakly to a fixed point of $f$.

%-----------------------------------------------------------------------------------------------------------------------------------------
\par
The purpose of this paper is to extend Xu's demiclosedness principle from asymptotically nonexpansive mappings in a uniformly convex Banach space to random asymptotically nonexpansive mappings in a complete random uniformly convex random normed module.

%-----------------------------------------------------------------------------------------------------------------------------------------
\par
To introduce the motivation and background for the work 
of this paper, let us first recall the historical and in 
particular the recent advances in random functional 
analysis.  Based on the idea of randomizing the traditional space theory of functional analysis, random functional analysis is developed as a whole random 
generalization of classical functional analysis \cite{Guo92, Guo93}, whose  basic frameworks are random normed modules (briefly, $RN$ modules) and random locally convex modules (briefly, $RLC$ modules). When the basic frameworks are endowed with the $(\varepsilon, 
\lambda)$-topology, they are not locally convex in 
general, so the traditional theory of  conjugate spaces is no longer universally applicable for the further development of random functional analysis, which leads to a systematic development of the theory of random conjugate spaces \cite{Guo2000, Guo08, Guo10a, GL05, GXX09}. In particular, the inherent connection between the random conjugate space $E^{*}$ of an $RN$ module $E$ and the  classical conjugate spaces $(L^{p}(E))^{'}$ of the abstract $L^{p}$-space $L^{p}(E)$ generated by $E$, namely, $(L^{p}(E))^{'}\cong L^{q}(E^{*})$, has played a crucial role in the development of the theory of random conjugate spaces, see \cite{Guo2000} or Proposition \ref{proposition2.6} of this paper for details. Random functional analysis is of fundamental importance in the development of nonsmooth differential geometry on metric measure space: since the notion of an $RN$ module is equivalent to that of an $L^{0}$-normed $L^{0}$-module, which was independently introduced by Gigli \cite{Gigli18} as a tool for the study of nonsmooth differential geometry, and it is also proved in \cite{GMT24} that every $L^{p}$-normed $L^{\infty}$-module introduced by Gigli \cite{Gigli18} is of the form $L^{p}(E)$ for a complete $RN$ module $E$ and the module dual of $L^{p}(E)$ is exactly $L^{q}(E^{*})$; see \cite{BPS2023, CGP2025, CLPV2025, LP, LPV} for the recent advances in nonsmooth differential geometry. In the last ten years, developing 
fixed  point theory of random functional analysis has 
been one of the central tasks with an aim of providing a 
tool for dynamic mathematical finance and stochastic 
differential equations \cite{Guo10a, Guo11, TMGC25,TMGY25}. Establishing the topological fixed theory of random functional analysis has motivated the development of the theory of stable compactness \cite{GWXYC25, TMG24, TMGC25,TMGY25}, while establishing its metric fixed point theory has stimulated the deep development of geometry of $RN$ modules \cite{GZWG20,MTG,SGT25}.

%-----------------------------------------------------------------------------------------------------------------------------------------

\par
Based on the recent advances in the fixed point theory of random functional analysis, in this paper we continue to study the random demiclosedness principle for a random asymptotically nonexpansive mapping. Let us again recall Xu's classical work \cite[Theorem 2]{X91}. Xu's demiclosedness principle \cite[Theorem 2]{X91} mainly depends on  \cite[Theorem 2]{B81}, which utilizes the theory of geometry of Banach spaces such as $B$-convexity  or convex approximation property. Although geometry of $RN$ modules early began with Guo and Zeng's work \cite{GZ10}, where it was proved that a complete $RN$ module $(E,\|\cdot\|)$ is random uniformly convex iff $L^{p}(E)$ is uniformly convex for any given $p\in(1,+\infty)$ (see also Proposition \ref{lemma2.3.3} of this paper for details), the random versions of the notions of the classical $B$-convexity and convex approximation property have been never presented and studied for an $RN$ module. Therefore, if one hopes to established the random demiclosedness principle in an $RN$ module, it is impossible to directly move Xu's approach \cite[Theorem 2]{X91} to the random setting. However, motivated by the works in \cite{GWXYC25, SGT25}, we can solve this problem by decomposing a random asymptotically nonexpansive mapping $f$ defined on an  almost surely  bounded closed $L^{0}$-convex subset $G$ of a complete random uniformly convex $RN$ module $(E,\|\cdot\|)$ into a  sequence  of smaller  mappings $\{f_{i},i\in \mathbb{N}\}$ such that each $f_{i}$ is an asymptotically nonexpansive mapping defined on  a bounded closed convex subset $G_{i}$ of the uniformly convex Banach space $L^{p}(E)$. Our success relies on two key connections: one is the connection between the random uniform convexity of $E$ and the classical uniform convexity, the other is the connection between the random conjugate space $E^{*}$ of $E$ and the classical conjugate space $(L^{p}(E))^{'}$ of $L^{p}(E)$. With this decomposition at hands, we can first apply Xu's classical demiclosedness principle \cite[Theorem 2]{X91} to each $f_{i}$ and then obtain the random demiclosedness of $f$ by combining  the $\sigma$-stable properties of $f$ and $G$.

\par
The remainder of this paper is organized as follows. In Section \ref{section2}, we provide some prerequisites and further present the main result --- Theorem \ref{theorem1.7}, namely, random demiclosedness principle for a random asymptotically nonexpansive mapping. Finally, Section \ref{section3} is devoted to the proof of Theorem \ref{theorem1.7}.

%% Use \subsection commands to start a subsection.

%####################################################################
%####################################################################
\section{Preliminaries and main result}\label{section2}
%####################################################################
%####################################################################

\par
Throughout this paper, $(\Omega,\mathcal{F},P)$ denotes a given probability space, $\mathbb{N}$ the set of positive integers, $\mathbb{K}$ the scalar field $\mathbb{R}$ of real numbers or $\mathbb{C}$ of complex numbers, $L^{0}(\mathcal{F},\mathbb{K})$ the algebra of equivalence classes of $\mathbb{K}$-valued $\mathcal{F}$-measurable random variables  on $(\Omega,\mathcal{F},P)$, $\bar{L}^{0}(\mathcal{F})$ the set of equivalence classes of  extended real valued $\mathcal{F}$-measurable random  variables  on $(\Omega,\mathcal{F},P)$ and $L^{0}(\mathcal{F}):=L^{0}(\mathcal{F},\mathbb{R})$.

\par
The partial order $\leq$ on $\bar{L}^{0}(\mathcal{F})$  is defined by  $\xi\leq \eta$ iff $\xi^{0}(\omega)\leq \eta^{0}(\omega)$ for almost surely all $\omega\in \Omega$ (briefly, $\xi^{0}(\omega)\leq\eta^{0}(\omega)$ $\mathbf{a.s.}$), where $\xi^{0}$ and $\eta^{0}$ are arbitrarily chosen representatives of $\xi$ and $\eta$,  respectively. Then, $(\bar{L}^{0}(\mathcal{F}),\leq)$ is a complete lattice \cite{D57}. Specially, $(L^0(\mathcal{F}), \leq)$, as a sublattice of  $(\bar{L}^{0}(\mathcal{F}),\leq)$, is Dedekin complete.

%---------------------------------------------------------------------------------------------------------------

\par
As usual, for any $\xi$ and $\eta$ in $\bar{L}^{0}(\mathcal{F})$, $\xi<\eta$ means $\xi\leq\eta$ and $\xi\neq\eta$ , whereas, for
any $A\in \mathcal{F}$, $\xi<\eta$ on $A$ means  $\xi^{0}(\omega)<\eta^{0}(\omega)$ for almost all $\omega\in A$, where $\xi^{0}$ and $\eta^{0}$ are respectively arbitrarily chosen representatives of $\xi$ and $\eta$.
\par
In this paper, the following notations are always employed:
\par $L^{0}_{+}(\mathcal{F})=\{\xi \in L^{0}(\mathcal{F})~|~\xi \geq 0\}$;
\par $L^{0}_{++}(\mathcal{F})=\{\xi \in L^{0}(\mathcal{F})~|~\xi>0~\text{on}~\Omega\}$.
\par
For any $A\in \mathcal{F}$, $\tilde{I}_{A}$ stands for the equivalence class of  $I_{A}$, where $I_{A}$ is the characteristic function of $A$, namely, $I_{A}(\omega)=1$ if $\omega\in A$ and $0$ otherwise.
\par

%--------------------------------------------------

    The notion of an $RN$ module was introduced independently by Guo in \cite{Guo92,Guo93} and by Gigli in \cite{Gigli18}. An ordered pair $(E,\| \cdot \|)$ is called an $RN$ module over  $\mathbb{K}$ with base $(\Omega,\mathcal{F},P)$ if $E$ is a left module over the algebra $L^0(\mathcal{F},\mathbb{K})$ (briefly, an $L^0(\mathcal{F},\mathbb{K})$-module$)$ and $\| \cdot \|$ is a mapping from $E$ to $L^0_+(\mathcal{F})$ such that the following  are satisfied:
	\begin{enumerate}[(RNM-1)]
		\item $\|x\| = 0$ implies $x = \theta$ (the null in $E$);
		\item $\| \xi\cdot x \| = |\xi| \cdot\|x\|$ for any $(\xi, x)\in L^{0}(\mathcal{F},\mathbb{K})\times E$;
		\item $\|x+y\| \leq \|x\| + \|y\|$ for all $x$ and $y \in E$.
	\end{enumerate}
	As usual, $\| \cdot \|$ is called the $L^0$-norm on $E$.

\par
When $(\Omega, \mathcal{F}, P)$ is trivial, an $RN$ module reduces to a normed space. The simplest nontrivial $RN$ is $(L^0(\mathcal{F},\mathbb{K}),|\cdot|)$, where $|\cdot|$ is the usual absolute value mapping on $L^0(\mathcal{F},\mathbb{K})$.

%---------------------------------------------------------------------------------------------------------------

\par
In this paper, an $RN$ module is always endowed with the $(\varepsilon, \lambda)$-topology as follows. Let $(E, \|\cdot\|)$ be an $RN$ module over $\mathbb{K}$ with base $(\Omega, \mathcal{F}, P)$. For any real numbers $\varepsilon > 0$ and $0 < \lambda < 1$, let $N_{\theta}(\varepsilon, \lambda) = \{x \in E~|~P\{\omega \in \Omega|~ \|x\|(\omega) < \varepsilon\} > 1 - \lambda\}$. Then $\mathcal{U}_{\theta} = \{N_{\theta}(\varepsilon, \lambda)~|~\varepsilon > 0, 0 < \lambda < 1\}$ forms a local base for some metrizable linear topology on $E$, called the $(\varepsilon, \lambda)$-topology, denoted by $\mathcal{T}_{\varepsilon,\lambda}$. According to \cite{Guo92,Guo10a}, the following statements hold:
	\begin{enumerate}[{\rm(1)}]
		\item $(L^{0}(\mathcal{F}, \mathbb{K}), \mathcal{T}_{\varepsilon,\lambda})$ is a topological algebra over $\mathbb{K}${\rm ;}
		\item $(E,\mathcal{T}_{\varepsilon,\lambda})$ is a topological module over the topological algebra $(L^{0}(\mathcal{F}, \mathbb{K}),\\ \mathcal{T}_{\varepsilon,\lambda})${\rm ;}
		\item A sequence $\{x_{n}, n \in \mathbb{N}\}$ in $E$ converges in the $\mathcal{T}_{\varepsilon,\lambda}$ to $x\in E$ iff $\{\|x_{n}-x\|, n \in \mathbb{N}\}$  converges in probability to $\theta${\rm .}
	\end{enumerate}

It is easy to see that the $\mathcal{T}_{\varepsilon,\lambda}$ for $(L^0(\mathcal{F},\mathbb{K}),|\cdot|)$ is exactly the topology of convergence in probability.

The idea of random conjugate spaces was introduced in \cite{Guo92} as follows. Let $(E,\|\cdot\|)$ be an $RN$ module over $\mathbb{K}$ with base $(\Omega, \mathcal{F}, P)$, a linear operator $f$ from $E$ to $L^{0}(\mathcal{F}, K)$ is called an almost surely (briefly, $\textbf{a.s.}$) bounded random linear functional if there exists some $\xi$ in $L^{0}_{+}(\mathcal{F})$ such that $|f(x)|\leq \xi\cdot\|x\|$ for any $x\in E$. It is well known from \cite{Guo92} that a linear operator $f: E\rightarrow L^{0}(\mathcal{F},\mathbb{K})$ is $\mathbf{a.s.}$ bounded if and only if $f$ is a continuous module homomorphism. Denote by $E^{*}$ the $L^{0}(\mathcal{F},\mathbb{K})$-module of continuous module homomorphisms from $E$ to $L^{0}(\mathcal{F}, \mathbb{K})$ and $\|\cdot\|^{*}: E^{*}\rightarrow L^{0}_{+}(\mathcal{F})$ by $\|f\|^{*}=\bigwedge\{\xi\in L^{0}_{+}(\mathcal{F})~|~|f(x)|\leq\xi\cdot \|x\|, \forall x\in E\}$ for any $f\in E^{*}$, then $(E^{*}, \|\cdot\|^{*})$ is again an $RN$ module over $\mathbb{K}$ with base $(\Omega, \mathcal{F}, P)$, called the random conjugate space of $(E,\|\cdot\|)$.

\par
For an $RN$ module $(E,\|\cdot\|)$, since $E$ is not necessarily locally convex, it makes no sense to speak of  its weak topology. With the help of the theory of random conjugate spaces, Guo \cite{Guo08} introduced the notion of random weak topology of $E$ as follows. For any real numbers $\varepsilon > 0$ and $0 < \lambda < 1$ and any $f\in E^{*}$, let $N_{\theta}(f, \varepsilon, \lambda)=\{x\in E~|~ P\{\omega\in\Omega~|~|f(x)|(\omega)<\varepsilon\}>1-\lambda\}$. Then $\mathcal{U}_{\theta}=\{\bigcap^{n}_{i=1}N_{\theta}(f_{i},\varepsilon_{i}, \lambda_{i})~|~\varepsilon_{i}>0, 0<\lambda_{i}<1, f_{i}\in E^{*},1\leq i\leq n$ and $n$ is any positive integer$\}$ forms a local base at $\theta$ of some Hausdorff linear topology on $E$, called the random weak topology of $E$, denoted by $\sigma(E,E^{*})$. Clearly, a net $\{x_{\alpha},\alpha\in\Gamma\}$ in $E$ converges in $\sigma(E,E^{*})$ to some element $x$ in $E$ iff the net $\{f(x_{\alpha}),\alpha\in\Gamma\}$ in $L^{0}(\mathcal{F},\mathbb{K})$ converges in probability $P$ to $f(x)$ for each $f\in E^{*}$.

\par
 The following basic connection between $RN$ modules and normed spaces will be used in this paper. Let $(E, \|\cdot\|)$ be an $RN$ module over $\mathbb{K}$ with base $(\Omega, \mathcal{F}, P)$ and $1\leq p\leq +\infty$. Let $L^{p}(E)=\{x\in E|~\|x\|_{p}<+\infty\}$, where $\|\cdot\|_{p}: E \rightarrow [0,+\infty]$ is defined by:
% $\|x\|_{p} =
% \begin{cases}
% 	\item \int_{\Omega}(\|x\|^{p})dP)^{\frac{1}{p}}, ~&\text{when}~1\leq p<+\infty;\\
% 	\item \inf\{M\in [0,+\infty], ~&\text{when}~p=+\infty
% \end{cases} $

	\[
\|x\|_{p} =
\begin{cases}
(\int_{\Omega}\|x\|^{p}dP)^{\frac{1}{p}}, & \text{when } 1\leq p<+\infty;\\
\inf\{M\in [0,+\infty]~|~\|x\|\leq M\}, & \text{when}~p=+\infty
\end{cases}
\]
	for all $x\in E$. Then, it is easy to see that $(L^{p}(E), \|\cdot\|_{p})$ is a normed space.  Moreover, if $(E,\|\cdot\|)$ is a complete $RN$ module, then $(L^{p}(E),\|\cdot\|_{p})$ is  a Banach space.

\par
 \par
 Proposition \ref{proposition2.6} below is an important representation theorem, which establishes the inherent connection between the random conjugate space $E^{*}$ of an $RN$ module $E$ and the classical conjugate space $(L^{p}(E))^{'}$ of $L^{p}(E)$.

%------------------------------------------------------------------------------------------------------------------

\begin{proposition}[\cite{Guo2000}]\label{proposition2.6}
$L^{q}(E^{*})\cong (L^{p}(E))^{'}$ under the canonical mapping $T$,  where $p$ and $q$ are a pair of H\"{o}lder conjugate numbers with $1\leq p<+\infty$, $(L^{p}(E))^{'}$ denotes the classical conjugate space of $L^{p}(E)$ and for each $F\in L^{q}(E^{*})$, $T_{F}$ (denoting $T(F))~:~L^{p}(E)\rightarrow \mathbb{K}$ is defined by $T_{F}(x)=\int_{\Omega}F(x)dP$,  for all $x\in L^{p}(E)$.
\end{proposition}

%------------------------------------------------------------------------------------------------------------------

%----------------------------------------------------------------------------------------------------------------------

\par
Let $(E, \|\cdot\|)$ be an $RN$ module over $\mathbb{K}$ 
with base $(\Omega, \mathcal{F}, P)$. A nonempty subset 
$G$ of $E$ is said to be $L^{0}$-convex,  if $\xi x+ 
\eta y\in G$ for any $x$, $y\in G$ and any $\xi$, 
$\eta\in L^{0}_{+}(\mathcal{F})$ with $\xi+\eta=1$;  
$\mathbf{a.s.}$ bounded or $L^{0}$-bounded, if 
$\bigvee\{\|x\|~|~x\in G\}\in L^{0}_{+}(\mathcal{F})$. The 
set $supp(E)=\{\omega\in 
E~|~\xi^{0}(\omega)=+\infty\}$ 
is called the support of $E$ ($supp(E)$ is unique 
$\mathbf{a.s}$.), where $\xi =\bigvee\{\|x\|~|~x\in E\}$ 
and $\xi^{0}$ is an arbitrarily chosen
representative of $\xi$. If $P(supp(E))=1$, then $E$ is said to have full support. In the remainder of this paper, it is always assumed that all $RN$ modules mentioned have full support.
\par
Further, we employ the following notations for a brief 
introduction to random uniformly convex $RN$ modules:
$$\small{\varepsilon_{\mathcal{F}}[0,2]=\{\varepsilon\in 
L^{0}_{++}(\mathcal{F})~|~\mbox{there exists a positive 
number}~\lambda~\mbox{such that}~\lambda \leq
\varepsilon \leq 2\}.}$$
$$\small{\delta_{\mathcal{F}}[0,1]=\{\delta\in 
L^{0}_{++}(\mathcal{F})~|~\mbox{there exists a positive 
number}~ \eta~\mbox{such that}~\eta \leq \delta \leq
1\}.}$$
\par
For any $x,~y$ in $E$, denote the equivalence class of $\mathcal{F}$-measurable set $\{\omega\in \Omega~|~\|x\|^{0}(\omega)\neq 0\}$ by $A_{x}$, called the support of $x$, where $\|x\|^{0}$ is an arbitrarily chosen representative of $\|x\|$; and we simply write $A_{x,y}=A_{x}\cap A_{y}$ and $B_{x,y}=A_{x}\cap A_{y}\cap A_{x-y}$.

%---------------------------------------------------------------------------------------------------------------

\begin{definition}[\cite{GZ10}]\label{definition1.4}
	Let $(E, \|\cdot\|)$ be a  complete $RN$ module over $\mathbb{K}$ with base $(\Omega, \mathcal{F},P)$. $E$ is said to be random uniformly convex if for each $\varepsilon\in \varepsilon_{\mathcal{F}}[0,2]$ there exists  $\delta\in \delta_{\mathcal{F}}[0,1]$ such that $\|x-y\|\geq \varepsilon $ on $D$ always implies $\|x+y\|\leq 2(1-\delta)$ on $D$ for any $x$, $y$ $\in$ $U(1)$ and any $D\in \mathcal{F}$ such that $D\subset B_{x, y}$ and $P(D)>0$, where $U(1)=\{z\in E~|~\|z\|\leq 1\}$, called the random closed unit ball of $E$.
\end{definition}

%-------------------------------------------------------

Geometry of $RN$ modules began with the works of Guo and Zeng \cite{GZ10, GZ12}, where  they also established the equivalence between random uniform convexity of an $RN$ module $E$ and ordinary uniform convexity of the abstract normed space $L^{p}(E)$.

\begin{proposition}[\cite{GZ10}]\label{lemma2.3.3}
Let $(E, \|\cdot\|)$ be a complete $RN$ module over $\mathbb{K}$ with base $(\Omega, \mathcal{F}, P)$ and $p$ be a given positive number with $1<p<+\infty$. Then $(L^{p}(E), \|\cdot\|_{p})$ is uniformly convex iff $E$ is random uniformly convex.	
\end{proposition}

%-------------------------------------------------------

We first studied random asymptotically nonexpansive mappings in \cite{SGT25}. Let ($E,\|\cdot\|$) be an $RN$ module over $\mathbb{K}$ with base $(\Omega,\mathcal{F},P)$ and  $G$ be a nonempty subset of $E$. A mapping $f:G\rightarrow G$  is said to be random asymptotically nonexpansive if there exists a sequence $\{\eta_{m}, m\in \mathbb{N}\}$ in $L^{0}_{+}(\mathcal{F})$ with $\{\eta_{m}, m\in \mathbb{N}\}$ convergent $\textbf{a.s.}$ to 1, such that
	\begin{align*}
		\| f^{m}x-f^{m}y\| \leq \eta_{m}\| x-y\|, \forall x,y \in G~and ~m\in \mathbb{N}.
	\end{align*}

%---------------------------------------------------------------------------------------------------------------------------

The  main result of this paper is Theorem  \ref{theorem1.7} below, which can be regarded as a random generalization of Xu's demiclosedness principle \cite[Theorem 2]{X91}.

%----------------------------------------------------------------------

\begin{theorem}\label{theorem1.7}
Let $(E,\|\cdot\|)$ be a complete random uniformly convex $RN$ module over $\mathbb{K}$ with base $(\Omega,\mathcal{F},P)$, $G$ an $\mathbf{a.s.}$ bounded closed  $L^{0}$-convex subset of $E$ and $f: G \rightarrow G$ a random asymptotically nonexpansive mapping. Then $(I-f)$ is random demiclosed at $\theta$, namely, for each sequence $\{x_{n}, n\in \mathbb{N}\}$ in $G$, if $\{x_{n}, n\in \mathbb{N}\}$ converges in $\sigma(E,E^{*})$ to $x$  and $\{(I-f)x_{n}, n\in\mathbb{N}\}$  converges to $\theta$, then  $(I-f)x=\theta$, where $I$ denotes the identity operator on $E$.
\end{theorem}

%-----------------------------------------------------

\section{Proof of the Theorem 2.4}\label{section3}

%-------------------------------------------------------------------------------------------------------------

To prove Theorem 2.10, let us first recall from \cite{Guo10a} the notions of stability of sets and mappings as follows. Let $(E, \|\cdot\|)$ be an $RN$ module over $K$ with base $(\Omega,\mathcal{F},P)$ and $G$ be a nonempty subset of $E$. $G$ is said to be finitely stable if $\tilde{I}_{A}x+ \tilde{I}_{A^{c}}y\in G$ for any $x,y\in G$ and any $A\in \mathcal{F}$, where $A^{c}=\Omega \setminus A$. $G$ is said to be $\sigma$-stable (or, to have the countable concatenation property in the original terminology of \cite{Guo10a}) if for any sequence $\{x_{n},n\in \mathbb{N}\}$ in $G$ and  any countable partition $\{A_{n},n\in \mathbb{N}\}$ of $\Omega$ to $\mathcal{F}$ (namely, each $A_{n}\in \mathcal{F}$, $A_{i}\bigcap A_{j}=\emptyset$ for any $i\neq j$, and $\bigcup^{\infty}_{n=1}A_{n}=\Omega$) there exists $x\in G$ ($x$ is unique and can be written as $\sum^{\infty}_{n=1}\tilde{I}_{A_{n}}x_{n}$) such that $\tilde{I}_{A_{n}}x=\tilde{I}_{A_{n}}x_{n}$ for each $n\in \mathbb{N}$. Further, a mapping $f:G\rightarrow G$ is said to be $\sigma$-stable, if $G$ is $\sigma$-stable and
	$$f(\sum^{\infty}_{n=1}\tilde{I}_{A_{n}}x_{n})=\sum^{\infty}_{n=1}\tilde{I}_{A_{n}}f(x_{n})$$
for every sequence $\{x_{n},n\in \mathbb{N}\}$ in $G$ and every countable partition $\{A_{n},n\in \mathbb{N}\}$ of $\Omega$ to $\mathcal{F}$.

%-------------------------------------------------------------------------------------------

 According to \cite[Lemma 2.11]{GZWG20}, if $G$ is a $\sigma$-stable subset of an $RN$ module and $f:G \to G$ is an asymptotically nonexpansive mapping, then $f$ is necessarily $\sigma$-stable.

\begin{lemma}\label{lemma2.5}
Let $(E,\|\cdot\|)$ be a complete $RN$ module over $K$ 
with base $(\Omega,\mathcal{F},P)$, $G$ a nonempty 
$\mathbf{a.s.}$ bounded closed $L^{0}$-convex subset of 
$E$ with $\theta\in G$, $f: G \rightarrow G$ a random 
asymptotically nonexpansive mapping and $p$ a given 
positive number with $1<p<+\infty$.  Then there exists 
a countable partition $\{A_{i}, i\in\mathbb{N}\}$ of 
$\Omega$ to $\mathcal{F}$  such that for each  $i\in 
\mathbb{N}$:
\begin{enumerate}[{\rm(1)}]
\item $\tilde{I}_{A_{i}}G$ is a nonempty bounded $\|\cdot\|_{p}$-closed convex subset of the Banach space $L^{p}(E)$ such that $\bigvee\{\|x\|~|~x\in \tilde{I}_{A_{i}}G\}\leq\ n_{i}$ for some $n_{i}\in \mathbb{N}$;
\item $f_{i}:\tilde{I}_{A_{i}}G\rightarrow \tilde{I}_{A_{i}}G$ defined by
$$f_{i}(\tilde{I}_{A_{i}}x)=\tilde{I}_{A_{i}}f(x), \forall x\in G$$
is an asymptotically nonexpansive mapping.
\end{enumerate}
\end{lemma}
%----------------------------------------------------------------------------------------------------------------------------
\begin{proof}
Since $f: G \rightarrow G$ is a random asymptotically 
nonexpansive mapping,  there exists a sequence 
$\{\eta_{m}, m\in \mathbb{N}\}$ in 
$L^{0}_{+}(\mathcal{F})$ with $\{\eta_{m}, m\in 
\mathbb{N}\}$ convergent $\textbf{a.s.}$ to 1, such that
$$\|f^{m}x-f^{m}y\parallel \leq \eta_{m}\| x-y\|,~ 
\forall x,y \in G,~m\in \mathbb{N}.$$
For each $k\in \mathbb{N}$, by  Egoroff's theorem there 
exists $E_{k}\in \mathcal{F}$ with $P(\Omega \setminus 
E_{k})<\frac{1}{k}$ such that $\{\eta_{m}, m\in 
\mathbb{N}\}$ converges uniformly to 1 on $E_{k}$.  It 
is clear that $P(\bigcup^{\infty}_{k=1}E_{k})=1$. We 
can, without loss of generality, assume that 
$\Omega=\bigcup^{\infty}_{k=1}E_{k}$. Moreover, let 
$\Omega_{1}=E_{1}$ and $$\Omega_{k}=E_{k}\setminus 
\bigcup_{i=1}^{k-1}E_{i-1}~\text{for~any}~k\in 
\mathbb{N}~\text{with}~k\geq2,$$
then $\{\Omega_{k},k\in\mathbb{N}$\} is a countable 
partition of $\Omega$ to $\mathcal{F}$ and $\{\eta_{m}, 
m\in \mathbb{N}\}$ converges uniformly to 1 on each 
$\Omega_{k}$.

%--------------------------------------------------------------------------------------------------------------------------

\par
For each $\Omega_{k}$, since $\{\eta_{m}, m\in \mathbb{N}\}$ converges uniformly to 1 on $\Omega_{k}$, there exists $m_{k}\in \mathbb{N}$ such that $$\tilde{I}_{\Omega_{k}}\eta_{m}\in L^{\infty}(\mathcal{F}),~\forall m\geq m_{k}.$$
Further, let $$H_{k,j}=\{\omega\in \Omega_{k}|j-1\leq  (\bigvee_{m=1}^{m_{k}}\eta_{m})^{0}(\omega)<j\}$$
for  any $j\in \mathbb{N}$, where $(\bigvee_{m=1}^{m_{k}}\eta_{m})^{0}$ is an arbitrarily chosen representative of $\bigvee_{m=1}^{m_{k}}\eta_{m}$.
Then $\{H_{k,j}, j\in \mathbb{N}\}$ is a countable partition of $\Omega_{k}$ to $\mathcal{F}$ such that  $$\tilde{I}_{H_{k,j}}\eta_{m}\in L^{\infty}(\mathcal{F}),~\forall m\in \mathbb{N}.$$

\par
Since $G$ is an $\textbf{a.s}$ bounded subset of $E$, there exists $\xi\in L^{0}_{+}(\mathcal{F})$ such that $\|g\|\leq \xi$ for all $g\in G$. Let
$$C_{n}=\{\omega\in\Omega|n-1\leq\xi^{0}(\omega)<n\}$$
for any $n\in \mathbb{N}$, where $\xi^{0}$ is an arbitrarily chosen representative of $\xi$. Then $\{C_{n}, n\in \mathbb{N}\}$ is a countable partition of $\Omega$ to $\mathcal{F}$.

%-------------------------------------------------------------------------------------------------------------

\par
It is clear that $\{C_{n}\cap H_{k,j}, n, k,j\in \mathbb{N}\}$ is a countable partition of $\Omega$ to $\mathcal{F}$, and we claim that it is the desired countable partition.

\par
(1) First, for any $x\in \tilde{I}_{C_{n}\cap H_{k,j}}G\subseteq G$, since $\|x\|\leq \xi$, we have
$$\|x\|=\|\tilde{I}_{C_{n}\cap H_{k,j}}x\|\leq \tilde{I}_{C_{n}\cap H_{k,j}}\xi\leq n,$$
implying
$$\|x\|_{p}=(\int_{\Omega}\|x\|^{p}dP)^{\frac{1}{p}}\leq (\int_{\Omega}n^{p}dP)^{\frac{1}{p}}=n.$$
Then  $\tilde{I}_{C_{n}\cap H_{k,j}}G$ is a bounded subset of $(L^{p}(E),\|\cdot\|_{p})$.

%--------------------------------------------------------------------------------------------------------------------------

\par
Second, since $\tilde{I}_{C_{n}\cap H_{k,j}} G$ is a closed subset of $E$, by the Lebesgue dominance convergence theorem $\tilde{I}_{C_{n}\cap H_{k,j}} G$ is a $\|\cdot\|_{p}$-closed subset of $(L^{p}(E),\|\cdot\|_{p})$.

\par
Third, it is easy to check that $\tilde{I}_{C_{n}\cap H_{k,j}} G$ is $L^{0}$-convex and hence naturally convex.

\par
(2) Since $f$ is $\sigma$-stable, then it is easy to check that the mapping $f_{n,j,k}$ defined by
$$f_{n,j,k}(\tilde{I}_{C_{n}\cap H_{k,j}}x)=\tilde{I}_{C_{n}\cap H_{k,j}}f(x), \forall x\in G$$
is well defined and maps $\tilde{I}_{C_{n}\cap H_{k,j}}G$ into $\tilde{I}_{C_{n}\cap H_{k,j}}G$.

\par
Let $\beta_{m}(n,j,k)=\|\tilde{I}_{C_{n}\cap H_{k,j}}\eta_{m}\|_{\infty}$ for any $m\in \mathbb{N}$, then $\lim_{m\rightarrow\infty}\beta_{m}(n,j,k)=1$ for any fixed $n, j$ and $k$, since $\{\eta_{m}, m \in \mathbb{N}\}$ converges uniformly to 1 on each $\Omega_{k}$. Furthermore, since
\begin{align*}
	f^{m}_{n,j,k}(\tilde{I}_{C_{n}\cap H_{k,j}}x)&=f^{m-1}_{n,j,k}(f_{n,j,k}(\tilde{I}_{C_{n}\cap H_{k,j}}x))\\
	&=f^{m-1}_{n,j,k}(\tilde{I}_{C_{n}\cap H_{k,j}}f(x))\\
    &=f^{m-2}_{n,j,k}(f_{n,j,k}(\tilde{I}_{C_{n}\cap H_{k,j}}f(x)))\\
    &=f^{m-2}_{n,j,k}(\tilde{I}_{C_{n}\cap H_{k,j}}(f^{2}(x)))\\
	&\cdots\\
	&=\tilde{I}_{C_{n}\cap H_{k,j}}f^{m}(x)
\end{align*}
for any $m\in \mathbb{N}$ and any $x\in G$, where $f^{m}_{n,j,k}$ and $f^{m}$ denote the $m$-th iteration  of $f_{n,j,k}$ and $f$, respectively, we have
\begin{align*}
	&\| f^{m}_{n,j,k}(\tilde{I}_{C_{n}\cap H_{k,j}}x)-f^{m}_{n,j,k}(\tilde{I}_{C_{n}\cap H_{k,j}}y)\|_{p}\\
    &=(\int_{\Omega}\| f^{m}_{n,j,k}(\tilde{I}_{C_{n}\cap H_{k,j}}x)-f^{m}_{n,j,k}(\tilde{I}_{C_{n}\cap H_{k,j}}y)\|^{p}dP)^{\frac{1}{p}}\\
	&=(\int_{\Omega}\tilde{I}_{C_{n}\cap H_{k,j}}\|f^{m}(x)-f^{m}(y)\|^{p}dP)^{\frac{1}{p}}\\
	&\leq(\int_{\Omega}\tilde{I}_{C_{n}\cap H_{k,j}}|\eta_{m}|^{p}\|x-y\|^{p}dP)^{\frac{1}{p}}\\
    &=(\int_{\Omega}\tilde{I}_{C_{n}\cap H_{k,j}}|\eta_{m}|^{p}\|\tilde{I}_{C_{n}\cap H_{k,j}}x-\tilde{I}_{C_{n}\cap H_{k,j}}y\|^{p}dP)^{\frac{1}{p}}\\
	&\leq\beta_{m}(n,j,k)\|\tilde{I}_{C_{n}\cap H_{k,j}}x-\tilde{I}_{C_{n}\cap H_{k,j}}y\|_{p}
\end{align*}
for any $x,y\in G$ and any $m \in \mathbb{N}$. Thus $f_{n,j,k}$ is an asymptotically nonexpansive mapping.
\end{proof}

%-------------------------------------------------------------------------------------------------------------------
\begin{remark}
For any $x\in G$,
$$f(x)=\sum^{\infty}_{i=1}\tilde{I}_{A_{i}}f(x)=\sum^{\infty}_{i=1}\tilde{I}_{A_{i}}(\tilde{I}_{A_{i}}f(x))=\sum^{\infty}_{i=1}\tilde{I}_{A_{i}}f_{i}(\tilde{I}_{A_{i}}x).$$
So, (2) of Lemma \ref{lemma2.5} implies that the random asymptotically nonexpansive self-mapping $f$ defined on an  $\mathbf{a.s.}$ bounded closed $L^{0}$-convex subset $G$ of a complete $RN$ module $E$ can be represented by countably concatenating a sequence $\{f_{i}, i\in\mathbb{N}\}$ of asymptotically nonexpansive self-mappings defined on the corresponding bounded $\|\cdot\|_{p}$-closed convex subset $\tilde{I}_{A_{i}}G$ of the Banach space $L^{p}(E)$, where $p$ is a given positive number with $1<p<+\infty$.
\end{remark}

%------------------------------------------------------------------------------------------------------------------

\begin{lemma}\label{lemma2.6}
Let $(E,\|\cdot\|)$ be an $RN$ module over $K$ with base 
$(\Omega,\mathcal{F},P)$, $p$ a given positive number 
with $1<p<+\infty$ and $G$ be an $\mathbf{a.s.}$ bounded 
subset of $L^{p}(E)$ such that $\bigvee\{\|x\|~|~ x\in 
G\}\leq k$ for some $k\in\mathbb{N}$. If $\{x_{n}, 
n\in\mathbb{N}\}$  is a sequence in $G$ that converges 
in $\sigma(E,E^{*})$ to $x\in G$, then $\{x_{n}, 
n\in\mathbb{N}\}$ converges in  $\sigma(L^{p}(E), 
(L^{p}(E))')$ to $x$.
\end{lemma}

%------------------------------------------------------------------------------------------

\begin{proof}
For any $H\in (L^{p}(E))'$, by Prosition \ref{proposition2.6} there exists $F\in L^{q}(E^{*})$ such that $T(F)=H$.
Since $\{x_{n}, n\in\mathbb{N}\}$ converges in $\sigma(E,E^{*})$ to $x\in G$, then $|F(x_{n})-F(x)|$ converges  to 0.
Further, since $F\in L^{q}(E^{*})$ and
$$|F(x_{n})-F(x)|\leq\|F\|^{*}\|x_{n}-x|\|\leq2k\|F\|^{*},$$
by the Lebesgue dominance convergence theorem we have
\begin{align*}
\lim_{n\rightarrow \infty}|H(x_{n})-H(x)|&=\lim_{n\rightarrow \infty}|T(F)(x_{n})-T(F)(x)|\\
&\leq\lim_{n\rightarrow \infty}\int_{\Omega}|F(x_{n})-F(x)|dP\\
&=0.
\end{align*}
Therefore, $\{x_{n},n\in\mathbb{N}\}$ converges in  $\sigma(L^{p}(E), (L^{p}(E))')$ to $x$.
\end{proof}

With the above preparations, we are now ready to prove Theorem \ref{theorem1.7}.

\begin{proof}[\textbf{Proof of Theorem \ref{theorem1.7}}]
	We can, without loss of generality, assume that $\theta \in G$ (otherwise, take an arbitrary  $u_{0}\in G$, and replace $G$ and $f$ with  $G^{'}=G-u_{0}$ and $f^{'}:G^{'}\rightarrow G^{'}$ defined by
$f^{'}(u)=f(u+u_{0})-u_{0},~\forall u\in G^{'}$).

\par
Assume that $\{x_{n}, n\in \mathbb{N}\}$ is a sequence in $G$ converging in $\sigma(E,E^{*})$ to $x\in E$ and $(I-f)x_{n}$ converges to $\theta$. First, we have $x\in G$ by random Mazur theorem \cite[Corollary 3.4]{GXX09}. Let $\{A_{i}, i\in \mathbb{N}\}$ and $\{f_{i}, i\in \mathbb{N}\}$ be the countable partition and mappings obtained as in Lemma \ref{lemma2.5}, respectively.

\par
For each $i\in N$, since $E$ is a topological module  over the topological algebra $L^{0}(\mathcal{F}, \mathbb{K})$, it is easy to check that $\{\tilde{I}_{A_{i}}x_{n}, n\in \mathbb{N}\}$ converges in $\sigma(E,E^{*})$ to $\tilde{I}_{A_{i}}x$. For each $x_{n}$, since
$$(I-f_{i})(\tilde{I}_{A_{i}}x_{n})=\tilde{I}_{A_{i}}x_{n}-f_{i}(\tilde{I}_{A_{i}}x_{n})=\tilde{I}_{A_{i}}x_{n}-\tilde{I}_{A_{i}}f(x_{n})=\tilde{I}_{A_{i}}(I-f)(x_{n}),$$
then $\{(I-f_{i})(\tilde{I}_{A_{i}}x_{n}), n\in \mathbb{N}\}$ converges to $\theta$, and hence converges in $\|\cdot\|_{p}$-topology to $\theta$ by the Lebesgue dominance convergence theorem.

\par
To sum up, for each $i\in \mathbb{N}$, $f_{i}$ is an asymptotically nonexpansive self-mapping defined on a nonempty bounded $\|\cdot\|_{p}$-closed convex subset $\tilde{I}_{A_{i}}G$ of the uniformly convex Banach space $L^{p}(E)$, $\{\tilde{I}_{A_{i}}x_{n}, n\in \mathbb{N}\}$ is a sequence in $\tilde{I}_{A_{i}}G$ converging in $\sigma(L^{p}(E), (L^{p}(E))')$ to $\tilde{I}_{A_{i}}x$, and such that $\{(I-f)(\tilde{I}_{A_{i}}x_{n}), n\in \mathbb{N}\}$ converges in $\|\cdot\|_{p}$-topology to $\theta$. By Xu's classical demiclosedness principle \cite[Theorem2]{X91}, $f_{i}(\tilde{I}_{A_{i}}x)=\tilde{I}_{A_{i}}x$.

\par
Then, we have
$$f(x)=\sum^{\infty}_{i=1}\tilde{I}_{A_{i}}f_{i}(\tilde{I}_{A_{i}}x)=\sum^{\infty}_{i=1}\tilde{I}_{A_{i}}x=x.$$
Thus, $(I-f)$ is random demiclosed at $\theta$.
\end{proof}

%--------------------------------------------------------------------------------------------------------------------------

\begin{remark}\label{remark2.7}
	When $(\Omega, \mathcal{F}, P)$ is trivial, namely $\mathcal{F}=\{\Omega, \emptyset\}$,  the complete random uniformly convex $RN$ module $(E,\|\cdot\|)$ reduces to a uniformly convex Banach space,  $G$ to a bounded closed convex subset of $E$ and $f$ to an asymptotically nonexpansive mapping, and then Theorem \ref{theorem1.7} becomes \cite[Theorem 2]{X91}, namely, Xu's  classical demiclosedness principle for an asymptotically nonexpansive mapping.
\end{remark}

\end{document}